\documentclass[11pt]{article}
\usepackage[left=2cm,right=2cm, top=2.4cm,bottom=2.4cm,bindingoffset=0cm]{geometry}
\usepackage{tikz, titlesec, tikz-cd}
\usetikzlibrary{calc,intersections,through,backgrounds}
\usepackage{amsmath, amssymb}
\usepackage{amsthm}
\usepackage{mathtools}
\usepackage{booktabs, makecell}
\usepackage{adjustbox}

\usepackage{enumitem}

\mathtoolsset{showonlyrefs}

\usepackage{array, hyperref}

\usetikzlibrary{positioning}
\usetikzlibrary{decorations.text}
\usetikzlibrary{decorations.pathmorphing}
\usetikzlibrary{decorations.markings}

\tikzset{snake it/.style={decorate, decoration=snake}}
\tikzset{zigzag it it/.style={decorate, decoration=zigzag}}

\usepackage[bottom]{footmisc}

\usepackage{subcaption}

\newlength\oversetwidth
\newlength\underwidth

\titleformat{\section}{\large\bfseries\filcenter}{\thesection}{1em}{}
\titleformat{\subsection}{\bfseries}{\thesubsection}{1em}{}
\titleformat{\subsubsection}[runin]{\bfseries}{\thesubsubsection}{1em}{}[.]

\usepackage{multirow}

\usetikzlibrary{arrows} 
\tikzset{
    >=stealth',
    pil/.style={
           ->,
           thick,
           shorten <=2pt,
           shorten >=2pt,}
}
\tikzset{->-/.style={decoration={
  markings,
  mark=at position .8 with {\arrow{>}}},postaction={decorate}}}

\newcommand{\C}{\mathbb{C}}
\newcommand{\R}{\mathbb{R}}
\renewcommand{\P}{\mathbb{P}}

\newcommand{\Id}{\mathrm{Id}}

\newcommand{\RP}{\mathbb{R}\mathrm{P}}

\renewcommand{\vec}{}

\usepackage[section]{placeins}

\newtheorem{lemma}{Lemma}[section]
\newtheorem{proposition}[lemma]{Proposition}
\newtheorem{theorem}[lemma]{Theorem}
\newtheorem{corollary}[lemma]{Corollary}

\theoremstyle{definition}
\newtheorem{remark}[lemma]{Remark}

\makeatletter
\renewenvironment{proof}[1][\proofname] {\par\pushQED{\qed}\normalfont\topsep6\p@\@plus6\p@\relax\trivlist\item[\hskip\labelsep\bfseries#1\@addpunct{.}]\ignorespaces}{\popQED\endtrivlist\@endpefalse}
\makeatother

\usepackage[margin=1cm]{caption}

\title{Intersecting the sides of a polygon}    
\author{Anton Izosimov\thanks{Department of Mathematics, University of Arizona, e-mail: \tt{izosimov@math.arizona.edu}}}
\date{}

\begin{document}
\maketitle
\abstract{
Consider the map $S$ which sends a planar polygon $P$ to a new polygon $S(P)$ whose vertices are the intersection points of second nearest sides of $P$. This map is the inverse of the famous pentagram map. In this paper we investigate the dynamics of the map $S$. Namely, we address the question of whether a convex polygon stays convex under iterations of $S$. Computer experiments suggest that this almost never happens. We prove that indeed the set of polygons which remain convex under iterations of $S$ has measure zero, and moreover it is an algebraic subvariety of codimension two. We also discuss the equations cutting out this subvariety, as well as their geometric meaning in the case of pentagons.
%
}


\section{Introduction}

Let $P$ be a planar polygon, and let $S(P)$ be the polygon whose vertices are the intersection points of second nearest sides of $P$, see Figure \ref{Fig:Smap}. The map $S$ is the inverse of the celebrated \textit{pentagram map}, defined by R.\,Schwartz \cite{schwartz1992pentagram} and studied by many others. Indeed, as can be seen in Figure \ref{Fig:Smap}, one can recover $P$ from $S(P)$ by intersecting consecutive {shortest} diagonals  (i.e.  diagonals connecting second-nearest vertices), which is precisely the definition of the pentagram map. In what follows, we denote the pentagram map by $D$, so that $S = D^{-1}$ and $D = S^{-1}$. 
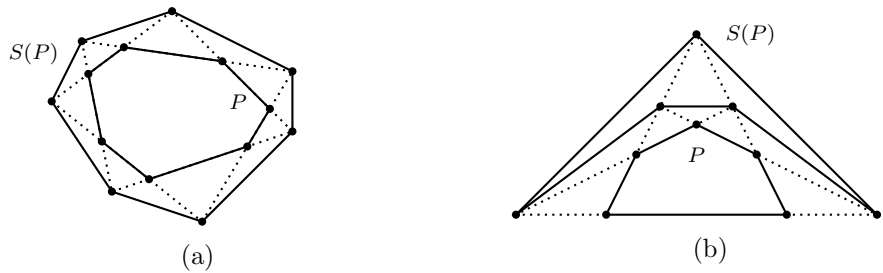
\begin{figure}[b]
\centering
  \begin{subfigure}{0.3\textwidth}
  \begin{tikzpicture}[ thick, scale = 0.8]
\coordinate (VK7) at (0,0);
\coordinate (VK6) at (1.5,-0.5);
\coordinate (VK5) at (3,1);
\coordinate (VK4) at (3,2);
\coordinate (VK3) at (1,3);
\coordinate (VK2) at (-0.5,2.5);
\coordinate (VK1) at (-1,1.5);
\draw[fill=black] (VK1) circle (.3ex);
\draw[fill=black] (VK2) circle (.3ex);
\draw[fill=black] (VK3) circle (.3ex);
\draw[fill=black] (VK4) circle (.3ex);
\draw[fill=black] (VK5) circle (.3ex);
\draw[fill=black] (VK6) circle (.3ex);
\draw[fill=black] (VK7) circle (.3ex);

\draw (VK7) -- (VK6) -- (VK5) -- (VK4) -- (VK3) -- (VK2) -- (VK1) -- cycle;
\draw [dotted, name path=AC] (VK7) -- (VK5);
\draw [dotted, name path=BD] (VK6) -- (VK4);
\draw [dotted ,name path=CE] (VK5) -- (VK3);
\draw [dotted, name path=DF] (VK4) -- (VK2);
\draw [dotted, name path=EG] (VK3) -- (VK1);
\draw [dotted,name path=FA] (VK2) -- (VK7);
\draw [dotted,name path=GB] (VK1) -- (VK6);

\path [name intersections={of=AC and BD,by=Bp}];
\path [name intersections={of=BD and CE,by=Cp}];
\path [name intersections={of=CE and DF,by=Dp}];
\path [name intersections={of=DF and EG,by=Ep}];
\path [name intersections={of=EG and FA,by=Fp}];
\path [name intersections={of=FA and GB,by=Gp}];
\path [name intersections={of=GB and AC,by=Ap}];

\draw  (Ap) -- (Bp) -- (Cp) -- (Dp) -- (Ep) -- (Fp) -- (Gp) -- cycle;

\draw[fill=black] (Ap) circle (.3ex);
\draw[fill=black] (Bp) circle (.3ex);
\draw[fill=black] (Cp) circle (.3ex);
\draw[fill=black] (Dp) circle (.3ex);
\draw[fill=black] (Ep) circle (.3ex);
\draw[fill=black] (Fp) circle (.3ex);
\draw[fill=black] (Gp) circle (.3ex);

\node at (-1.3,2.3) () {$\scriptstyle{S(P)}$};
\node at (2.1,1.5) () {$ \scriptstyle P$};

\end{tikzpicture}
    \caption{} \label{fig:1a}
  \end{subfigure}%
  \qquad\qquad
  \begin{subfigure}{0.3\textwidth}
  \begin{tikzpicture}[ thick, scale = 0.8]
\coordinate (A) at (0,0);
\coordinate (B) at (3,0);
\coordinate (C) at (2.5,1);
\coordinate (D) at (1.5,1.5);
\coordinate (E) at (0.5,1);

\coordinate (Ap) at (-1.5,0);
\coordinate (Bp) at (0.9,1.8);
\coordinate (Cp) at (2.1,1.8);
\coordinate (Dp) at (4.5,0);
\coordinate (Ep) at (1.5,3);

\draw [dotted] (A) -- (Ap) -- (E) -- (Ep) -- (C) -- (Dp) -- (B);
\draw [dotted] (Bp) --(D)  -- (Cp);

\draw  (A) -- (B) -- (C) -- (D) -- (E) -- cycle;
]
%
\draw   (Ap) -- (Bp) -- (Cp) -- (Dp) -- (Ep) -- cycle;
\draw[fill=black] (Ap) circle (.3ex);
\draw[fill=black] (Bp) circle (.3ex);
\draw[fill=black] (Cp) circle (.3ex);
\draw[fill=black] (Dp) circle (.3ex);
\draw[fill=black] (Ep) circle (.3ex);
\draw[fill=black] (A) circle (.3ex);
\draw[fill=black] (B) circle (.3ex);
\draw[fill=black] (C) circle (.3ex);
\draw[fill=black] (D) circle (.3ex);
\draw[fill=black] (E) circle (.3ex);
\node at (2.4,3) () {$\scriptstyle S(P)$};
\node at (1.5,1) () {$\scriptstyle P$};

\end{tikzpicture}
    \caption{} \label{fig:1a}
\end{subfigure}

\caption{The inverse pentagram map.}\label{Fig:Smap}
\end{figure}

\begin{figure}[t]
\centering
\includegraphics[width = 3cm]{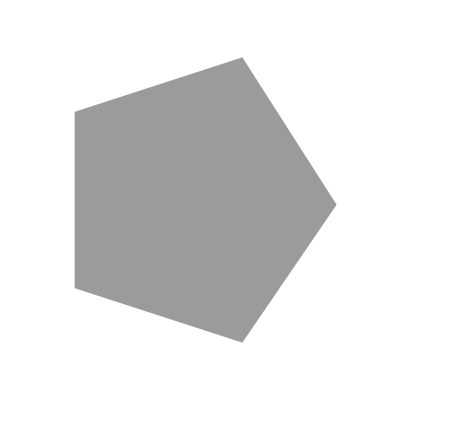}\includegraphics[width = 3cm]{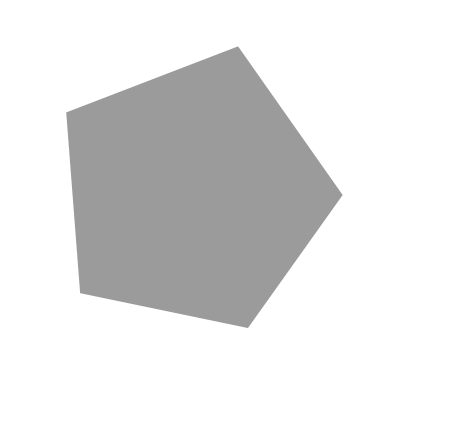}\includegraphics[width = 3cm]{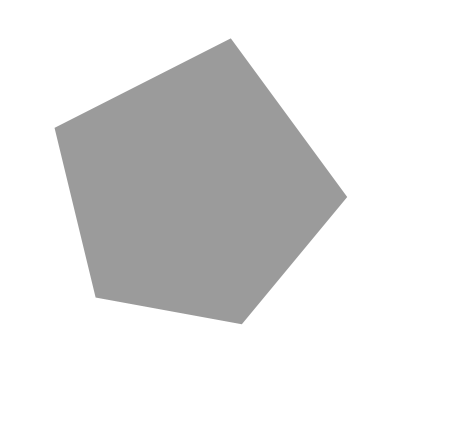}\includegraphics[width = 3cm]{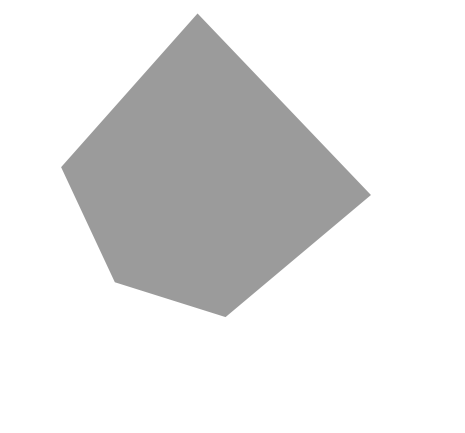}\includegraphics[width = 3cm]{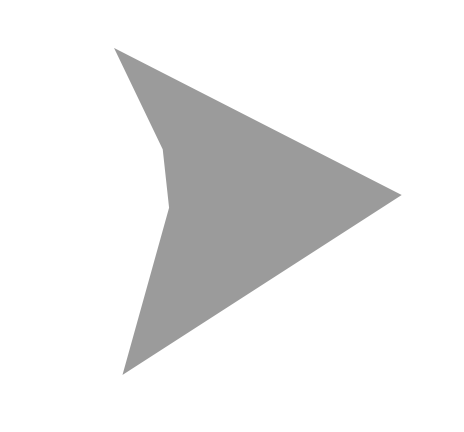}
\caption{The orbit of a polygon under the inverse pentagram map. Each iteration is normalized by means of an appropriate affine transformation.}\label{Fig:SmapOrbit}
\end{figure}

\par


As can be seen from Figure \ref{Fig:Smap}, the image of a convex polygon under the inverse pentagram map $S$ does not have to be convex. Moreover, computer experiments show that a randomly chosen polygon becomes non-convex after several applications of $S$, see Figure \ref{Fig:SmapOrbit}. The goal of the present paper is to determine necessary and sufficient conditions on a convex polygon $P$ which guarantee that {all} successive images of $P$ under  $S$, i.e. $S(P), S(S(P)), \dots$, are convex polygons. Our main result is that the set of polygons with this property has zero measure and moreover is a codimension two algebraic surface. Furthermore, we present explicit equations describing this surface, i.e. explicit conditions on $P$ which ensure that all its successive images under the map $S$ are convex. \par

We will state our main result in two different forms: in terms of a certain vector $d_P$ associated with the polygon $P$, and in terms of a certain operator $G_P$ also associated with $P$. \par To define $d_P$, consider a convex planar $n$-gon $P$ in the affine plane. Assume that the vertices of $P$ are labeled in cyclic order by residues modulo $n$. Denote by $ {\vec d}_i$ the vector connecting the vertices $i - 1$ and $i+1$, and let $\mathcal A_i$ be the area of the triangle cut out by the line joining those vertices, see Figure~\ref{Fig:ddef}. 
\begin{figure}[b]
\centering
\begin{tikzpicture}[thick, label distance=0.5mm, scale = 0.7]
\coordinate (VK7) at (0,0.5);
\coordinate (VK6) at (1.5,-0);
\coordinate (VK5) at (3,1);
\coordinate [label=right:$\scriptstyle{i + 1}$] (VK4) at (3,2);
\coordinate  [label=$ \scriptstyle i$] (VK3) at (1.5,3.5);
\coordinate[label=left:$\scriptstyle{i - 1}$]  (VK2) at (-0.5,2.5);
\coordinate (VK1) at (-1,1.5);
\draw[fill=black] (VK2) circle (.3ex);
\draw[fill=black] (VK3) circle (.3ex);
\draw[fill=black] (VK4) circle (.3ex);

\draw  (VK5) -- (VK4) -- (VK3) -- (VK2) -- (VK1) ;
\draw [dashed] (VK5) -- (VK6);
\draw [dashed] (VK1) -- (VK7);
\draw  [ ->-] (VK2) -- (VK4);
\fill [opacity = 0.2] (VK2) -- (VK3) -- (VK4) -- cycle;
\node at (1.3, 1.9) () {$\scriptstyle \vec{d}_i$};
\node at (1.4, 2.8) () {$\scriptstyle \mathcal A_i$};
\node at (7,1.5) () {$\displaystyle \vec d_P : = \sum_{i=1}^n \frac{\vec{d}_i}{\mathcal A_i}$};
%
%
%
%
%

\end{tikzpicture}
\caption{To the construction of the vector $\vec d_P$.}\label{Fig:ddef}
\end{figure}
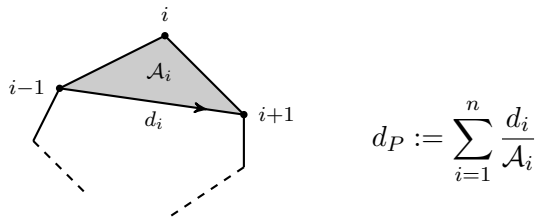
Then $d_P$ is defined by
\begin{align}\label{eq:d}
\vec d_P : = \sum_{i=1}^n \frac{\vec{d}_i}{\mathcal A_i}.
\end{align}
The first way to formulate our main result is as follows: for a convex polygon $P$, {all} successive images of $P$ under the inverse pentagram map  $S$ are convex if and only if $\vec d_P = 0$. Since this is equivalent to two algebraic equations on the coordinates of  vertices, it follows that the set of convex $n$-gons $P$ such that $S^k(P)$ is convex for every $k \geq 0$ is a codimension two algebraic surface in the $2n$-dimensional space of all convex $n$-gons. \par
Another way to state our result is in terms of a certain operator introduced by M.\,Glick \cite{glick2020limit}. This operator is defined as follows. Assume that we are given a polygon $P$ in the projective plane $\P^2$. Let $v_i \in \R^3$ be the vector of homogeneous coordinates of the $i$'th vertex of $P$. Then {Glick's operator} $ G_P \colon \R^3 \to \R^3$ is defined by 
\begin{align}\label{eq:glick}
G_P(v) := n v - \sum_{i=1}^n \frac{v_{i-1}\wedge v \wedge v_{i+1}}{v_{i-1}\wedge v_i \wedge v_{i+1}}\,v_i.
\end{align}
Observe that the right-hand side does not change under rescaling of $v_i$'s, so this is indeed a well-defined operator. Furthermore, as any operator in the $3$-space, Glick's operator can be interpreted as a projective mapping $\P^2 \to \P^2$.

\begin{theorem}\label{thm1}
Consider a convex planar polygon $P$ in the affine plane with at least five vertices. Then the following conditions are equivalent:
\begin{enumerate}
\item All successive images of the polygon $P$ under the inverse pentagram map $S$ are convex.
\item The associated vector $\vec d_P $ defined by \eqref{eq:d} is zero. 
\item The associated Glick's operator $G_P$ defined by \eqref{eq:glick}, regarded as a projective mapping, is affine. 
\end{enumerate}
\end{theorem}
We prove this theorem by establishing equivalences $2 \Leftrightarrow 3$ and $1 \Leftrightarrow 3$. The proof of $2 \Leftrightarrow 3$ is a direct computation which we outline in Section \ref{sec:gen1}. As for $1 \Leftrightarrow 3$, that part is based on more subtle properties of Glick's operator and its connection with pentagram dynamics. We discuss these properties in Section~\ref{sec:glick}, after which we complete the proof in Section \ref{sec:gen2}.

\begin{remark}\label{rm:c2}
Since the condition $\vec d_P = 0$ is equivalent to two algebraic equations in terms of vertices of $P$, it follows from Theorem \ref{thm1} that the set of polygons which remain convex under iterations of $S$ is an algebraic subvariety of codimension two. The latter can also be proved without using the vector  $\vec d_P$, as outlined in Remark \ref{rem:ref}.
\end{remark}

In the case of pentagons, there is one more condition equivalent to the above three:
\begin{theorem}\label{thm2}
For pentagons, the three conditions of Theorem \ref{thm1} are equivalent to the following one: \begin{enumerate}\setcounter{enumi}{3} \item The conic inscribed in $P$ and the conic circumscribed about $P$ are concentric. \end{enumerate}
\end{theorem} 
Apparently, there should be an elementary way of proving this by establishing equivalence $2 \Leftrightarrow 4$:  the inscribed and circumscribed conics are concentric if and only if $\vec d_P = 0$. However, 
 we are not aware of such a proof. What we do instead is directly prove the equivalence $1 \Leftrightarrow 4$.  The main ingredient of the proof is E.\,Kasner's theorem on pentagons. 
\begin{remark}
S.\,Tabachnikov \cite{tabachnikov2019kasner} proved that Kasner's theorem holds for all Poncelet polygons (i.e.  polygons inscribed in conic and circumscribed about a conic). Therefore, Theorem \ref{thm2} should be true for such polygons too. We do not consider this case here so as to not encumber the exposition.
\end{remark}
\begin{remark}
Glick (personal communication) suggested the following geometric interpretation of the condition $d_P = 0$, valid regardless of the number of vertices. Say that two polygons $P$ in $Q$ are \textit{dual with respect to the line at infinity} if sides of $P$ are parallel to shortest diagonals of $Q$, while sides of $Q$ are parallel to shortest diagonals of $P$, see Figure \ref{fig:dual}. Then it turns out that a polygon $P$ admits such a dual if and only if $d_P = 0$.
%
%

\end{remark}

  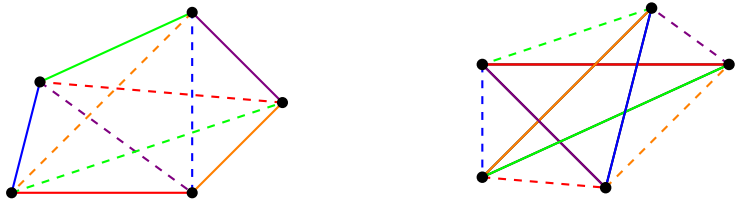
\begin{figure}[t]
  \centering
  \begin{tikzpicture}[ thick]
  \node at (0,0)
  {
    \begin{tikzpicture}[ thick, scale = 1.2]
    \coordinate  [label ={[shift={(0,0)}]225:$\scriptstyle $}](A) at (0,0);
\coordinate [label ={[shift={(0,0)}]315:$\scriptstyle $}] (B) at (2,0);
\coordinate  [label ={[shift={(0,0)}]0:$\scriptstyle $}]  (C) at (3,1);
\coordinate [label ={[shift={(0,0)}]90:$\scriptstyle $}]  (D) at (2,2);
\coordinate [label ={[shift={(0,0)}]135:$\scriptstyle $}]  (E) at ($( {7/4 - sqrt(33)/4)},  {3/4 + sqrt(11/3)/4)})$);
\draw [red] (A) -- (B);
\draw [orange] (B) -- (C);
\draw [violet] (C) -- (D);
\draw [green](D) -- (E);
\draw [blue](E) -- (A);
\draw [red, dashed] (C) -- (E);
\draw [orange, dashed] (D) -- (A);
\draw [violet, dashed] (E) -- (B);
\draw [green, dashed](A) -- (C);
\draw [blue, dashed](B) -- (D);
\draw[fill=black] (A) circle (.3ex);
\draw[fill=black] (B) circle (.3ex);
\draw[fill=black] (C) circle (.3ex);
\draw[fill=black] (D) circle (.3ex);
\draw[fill=black] (E) circle (.3ex);
      \end{tikzpicture}
      };
      
        \node at (6,0)
  {
    \begin{tikzpicture}[ thick, scale = 0.75]
\coordinate (Ap) at (0,0);
\coordinate (Bp) at (-3,-1);
\coordinate (Cp) at (-3,-3);
\coordinate (Dp) at ($( {-12 / (9 + sqrt(33))}, {1/4 * (-7 - sqrt(33))})$);
\coordinate (Ep) at ($( {1/2 * (-3 + sqrt(33))}, -1)$);
\draw [red, dashed] (Cp) -- (Dp) node[midway,below] {\color{black} $\scriptstyle $};
\draw [orange, dashed] (Dp) -- (Ep) node[midway,below] {\color{black} $\scriptstyle $};
\draw [violet, dashed] (Ep) -- (Ap) node[midway,above] {\color{black} $\scriptstyle $};
\draw [green, dashed](Ap) -- (Bp) node[midway,above] {\color{black} $\scriptstyle $};
\draw [blue, dashed](Bp) -- (Cp) node[midway,left] {\color{black} $\scriptstyle $};
\draw []  (Ap) -- (Cp) -- (Ep) -- (Bp) -- (Dp) -- cycle;
\draw [red]  (Ep) -- (Bp);
\draw [orange]  (Ap) -- (Cp);
\draw [violet]  (Bp) -- (Dp);
\draw [green]  (Cp) -- (Ep);
\draw [blue]  (Dp) -- (Ap);
\draw[fill=black] (Ap) circle (.5ex);
\draw[fill=black] (Bp) circle (.5ex);
\draw[fill=black] (Cp) circle (.5ex);
\draw[fill=black] (Dp) circle (.5ex);
\draw[fill=black] (Ep) circle (.5ex);
      \end{tikzpicture}
      };

%
]
%
%
%

\end{tikzpicture}
    \caption{Polygons dual with respect to the line at infinity. Lines of the same color and style are parallel to each other.} \label{fig:dual}
\end{figure}


{\bf Acknowledgments.} The author is grateful to 
Niklas Affolter, Max Glick, Boris Khesin, Richard Schwartz, Sergei Tabachnikov, and the anonymous referee for fruitful discussions and useful remarks. Figure \ref{Fig:SmapOrbit} was created with help of an applet written by Richard Schwartz. Figures~\ref{Fig:rmap}, \ref{Fig:Imap},  and \ref{Fig:Imap2} were created with help of software package Cinderella.  
 This work was supported by NSF grant DMS-2008021.
 \medskip
\section{When is Glick's operator affine?}\label{sec:gen1}
We start the proof of Theorem \ref{thm1} by showing the equivalence $2 \Leftrightarrow 3$: Glick's operator's $G_P$ is affine if and only $\vec d_P = 0$. Given a polygon in the affine plane, denote by $(x_i, y_i)$ the Cartesian coordinates of its vertices. We assume that the vertices are labeled in the counter-clockwise order. As in the introduction, denote by  $\vec d_i$ the vector connecting the vertices $i - 1$ and $i+1$, and let $\mathcal A_i$ be the area of the triangle cut out by the line joining those vertices, as shown in Figure~\ref{Fig:ddef}. Then a straightforward calculation shows that Glick's operator $G_P$ is given by
\begin{align}\label{eq:gpaff}
G_P\left(\begin{array}{c}x\\y \\z\end{array}\right) = n \left(\begin{array}{c}x \\y \\z\end{array}\right) - \sum_{i=1}^n
 \frac{1}{\mathcal A_i}
  \left( 
   \det
   \left(
   \begin{array}{cc}  \begin{array}{c} x \\y \end{array}& \vec d_i \end{array}
   \right)  
   -  z\cdot\det \left(\begin{array}{cc}x_{i-1} & x_{i+1} \\y_{i-1} &  y_{i+1} \end{array} \right) 
   \right)
  \left(\begin{array}{c}x_i \\y_i \\1\end{array}\right)
\end{align}
From this formula it follows that the preimage of the line at infinity under $G_P$ is given in homogeneous coordinates $x$, $y$, $z$ by 
\begin{align}\label{eq:inftypre}
 \det \left(\begin{array}{cc}  \begin{array}{c} x \\y \end{array}& \vec d_P \end{array}\right) = z\left(n+\sum_{i=1}^n  \frac{1}{\mathcal A_i} \det \left(\begin{array}{cc}x_{i-1} & x_{i+1} \\y_{i-1} &  y_{i+1} \end{array}\right)\right).
\end{align}
Now assume that $G_P$ is affine. Then the line at infinity is mapped to itself, so the above equation implies $\vec d_P = 0$. Conversely, assume that  $\vec d_P = 0$. Then the coefficient of $z$ in the above equation is easily seen to be invariant under coordinate transformations, and by choosing coordinates in such a way that the origin is in the interior of the polygon, one shows that the coefficient of $z$ is always positive and in particular does not vanish. At the same time, since $\vec d_P = 0$, the coefficients of $x$ and $y$ vanish, so the preimage of the line at infinity is the line at infinity, as desired.

\medskip

\section{Glick's operator and the pentagram map}\label{sec:glick}
In this section we discuss properties of Glick's operator $G_P$ and its connection with the dynamics of the pentagram map $D$. This is mainly an overview of \cite{glick2020limit}, but for some of the results we present a refined version. 
\begin{proposition}\label{prop:glick}
Let $P$ be a convex polygon in the affine plane with at least five vertices. Then the following is true:
\begin{enumerate}
\item For the dual polygon $P^*$, one has $G_{P^*} = G_P^*$ (recall that the dual polygon is the polygon in the dual projective plane whose vertices are the sides of initial polygon). 
\item Let $\mathrm{conv}(P)$ be the convex hull of vertices of $P$, and let $\mathrm{int}(\mathrm{conv}(P))$ be its interior. Then $G_P(\mathrm{conv}(P)) \subset \mathrm{int}(\mathrm{conv}(P))$. In particular, the projective mapping $G_P$ is well-defined at all points of $\mathrm{conv}(P)$ (i.e. none of those points belong to the kernel of $G_P$ when the latter is regarded as an operator in the $3$-space).  
\item The intersection
$
\bigcap_{k \geq 0} \mathrm{conv}(D^k(P)) = \bigcap_{k \geq 0} \mathrm{int}(\mathrm{conv}(D^k(P)))
$
is a single point, known as \textit{the limit point of} the pentagram orbit $D^k(P)$. This point is a fixed point of $G_P$. Moreover, it is the only fixed point of $G_P$ in $\mathrm{conv}(P)$.
\end{enumerate}
\end{proposition}
\begin{proof}
\begin{enumerate}
\item See \cite[Proposition 3.3]{glick2020limit}.
\item See \cite[Proposition 4.1]{glick2020limit}. The statement of that proposition is that $G_P(\mathrm{conv}(P)) \subset \mathrm{conv}(P)$, but it is actually proved that $G_P(\mathrm{conv}(P)) \subset \mathrm{int}(\mathrm{conv}(P))$. 
\item The existence of the limit point follows from \cite[Theorem 3.1]{schwartz1992pentagram}. Denote that point by $X$. Then, by \cite[Proposition 1.2]{glick2020limit}, we have $G_P(X) = X$. So it remains to show that there are no other fixed points in $\mathrm{conv}(P)$. Assume that $Y \neq X$ is another fixed point, $G_P(Y) = Y$. Note that $Y$ must be in the interior of $P$, since  $G_P(\mathrm{conv}(P)) \subset \mathrm{int}(\mathrm{conv}(P))$. Let $L$ be the line in $\RP^2$ through $X$ and $Y$. Then $G_P$ restricts to and defines a Moebius transformation of $L$. Furthermore, $G_P^2$ preserves each of the two connected components of $L \setminus \{X,Y\}$. Now, let $W$ and $Z$ be the intersection points of $L$ with the boundary of the polygon, as shown in Figure \ref{Fig:limitpoint}. 
\begin{figure}[b]
\centering
\begin{tikzpicture}[scale = 0.7, label distance = 0mm, thick]
\coordinate (A) at (2.25, 1.5);
\coordinate (B) at (3,3);
\coordinate  [] (C) at (1.5,4.5);
\coordinate (D) at (-0.75,3.75);
\coordinate (E) at (-0.75, 2.25);
\coordinate [label =$\scriptstyle X$] (X) at (0.75,3);
\coordinate [label =$\scriptstyle Y$] (Y) at (1.95,2.7);
\coordinate [label =$\scriptstyle L$] () at (4.2,2);

\draw  [ name path = poly]   (A) -- (B) -- (C) -- (D) -- (E) -- cycle;
\path [name path = AB] (A) -- (B);
\path [name path = DE] (D) -- (E);

%

\draw[fill=black] (A) circle (.3ex);
\draw[fill=black] (B) circle (.3ex);
\draw[fill=black] (C) circle (.3ex);
\draw[fill=black] (D) circle (.3ex);
\draw[fill=black] (E) circle (.3ex);
\draw[fill=black] (X) circle (.3ex);
\draw[fill=black] (Y) circle (.3ex);
\draw [color = black, name path=XY] ($(X)!-2cm!(Y)$) -- ($(Y)!-2cm!(X)$);
\path [name intersections={of=AB and XY,by={Z}}];
\path [name intersections={of=DE and XY,by={W}}];
\draw[fill=black] (W) circle (.3ex);
\draw[fill=black] (Z) circle (.3ex);
\coordinate [label ={[shift={(0.1,0)}]225:$\scriptstyle W$}] () at (W);
\coordinate [label ={[shift={(-0.2,0)}]300:$\scriptstyle Z$}] () at (Z);
%
%
\end{tikzpicture}
\caption{To the proof of Proposition \ref{prop:glick}.}\label{Fig:limitpoint}
\end{figure}
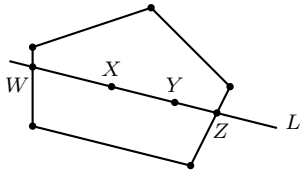
Then, since $G_P(\mathrm{conv}(P)) \subset \mathrm{int}(\mathrm{conv}(P))$, it follows that $G_P^2$ maps $W$ inside the interval $WX$ and $Z$ inside the interval $YZ$. At the same time, this Moebius transformation preserves $X$ and $Y$. But Moebius transformations with these properties do not exist. Indeed, for any non-trivial Moebius transformation which has two fixed points and preserves both of the intervals between those points, one of the fixed points must be attractive and the other one repelling. So, if $W$ is moved by $G_P^2$ towards $X$, then $Z$ should be moved away from $Y$. Thus, it is indeed not possible that $G_P$ has two fixed points in $\mathrm{conv}(P)$.\qedhere

\end{enumerate}
\end{proof}

\medskip

 \section{Convexity persists if and only if Glick's operator is affine}\label{sec:gen2}
We now prove the equivalence $1 \Leftrightarrow 3$: for a convex polygon $P$, all its successive images under the inverse pentagram map $S$ are convex if and only if the projective mapping $G_P$ is affine. To that end, we will reformulate the former condition in terms of the dual polygon. The following is well-known.
 \begin{proposition}\label{prop:duality}
 Projective duality intertwines the inverse pentagram map $S$ with the direct pentagram map $D = S^{-1}$. In other words,
 $
 S(P)^* = D(P^*).
 $
 \end{proposition}
 \begin{proof}
 The sides of $ S(P)^*$ are vertices of $S(P)$, which, by definition of $S$, are intersections of second-nearest sides of $P$. At the same time, the sides of $D(P^*)$ are shortest diagonals of $P^*$, i.e. diagonals connecting second nearest vertices. But second nearest vertices of $P^*$ are second nearest sides of $P$, so the diagonals connecting them are precisely the intersections of second-nearest sides of $P$. So, $ S(P)^* $ and $D(P^*)$ have the same sides and hence coincide, as claimed.
 \end{proof}
We now reformulate the convexity condition for  the polygon $S^k(P)$ in terms of its dual $D^k(P^*)$. Note that it is in general not true that the dual of a convex polygon is convex. In fact, this statement does not even make sense, since the definitions of convexity in the initial and dual planes require different structures. To define convexity in the initial plane, one needs to pick an affine chart, which is determined by choosing a \textit{line}. Likewise, convexity in the dual plane also becomes well-defined upon choice of a line, i.e. a \textit{point} in the initial plane.
So, in order to be able to talk about convexity in both the initial and dual planes, in the initial plane one needs to pick a line (which defines an affine chart) and a point (which defines an affine chart in the dual plane). The following is a folklore result. 
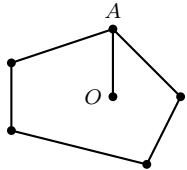
\begin{figure}[b]
\centering
\begin{tikzpicture}[thick, scale =0.9]
\coordinate (A) at (1.5, 1);
\coordinate (B) at (2,2);
\coordinate  [label=above:$\scriptstyle A$] (C) at (1,3);
\coordinate (D) at (-0.5,2.5);
\coordinate (E) at (-0.5, 1.5);
\coordinate [label=left:$\scriptstyle O$] (O) at (1,2);

\draw  (A) -- (B) -- (C) -- (D) -- (E) -- cycle;

%

\draw[fill=black] (A) circle (.3ex);
\draw[fill=black] (B) circle (.3ex);
\draw[fill=black] (C) circle (.3ex);
\draw[fill=black] (D) circle (.3ex);
\draw[fill=black] (E) circle (.3ex);
\draw[fill=black] (O) circle (.3ex);
\draw    (C) -- (O);
%
%
\end{tikzpicture}
\caption{To the proof of Lemma \ref{lemma:convduality}.}\label{Fig:convdual}
\end{figure}

\begin{lemma}\label{lemma:convduality}
Consider a projective plane $\P^2$ with a fixed line $L$ and point $O$. Assume that $P$ is a polygon in that plane which is convex in the affine chart $\P^2 \setminus L$ and contains the point $O$ in its interior. Then the dual polygon $P^*$ is convex in the affine chart  $(\P^2)^* \setminus O$ and contains the point $L$ in ins interior. 
\end{lemma}
\begin{remark}
The choice of a point $O$ turns $\P^2 \setminus L$ into a vector space. In that setting, the lemma can be reformulated by saying that the dual of a convex polygon containing the origin is also a convex polygon containing the origin. This is in fact true for any convex set, with an appropriate definition of duality.
\end{remark}
\begin{proof}[Proof of Lemma \ref{lemma:convduality}]
Figure \ref{Fig:convdual} shows a polygon $P$ in the affine chart $\P^2 \setminus L$. Let $A$ be an arbitrary vertex of $P$. Consider the straight line interval connecting $O$ to $A$. Then none of the sides of $P$ non-adjacent to $A$ intersect that interval, and neither does the line at infinity $L$. Therefore, all the sides of $P$ non-adjacent to $A$, as well as the line $L$ belong to the same connected component in the space of lines in $\P^2 \setminus \{O,A\}$. But this means that all vertices of $P^*$ not adjacent to the side $A$, as well as the point $L$, belong to the same connected component of $(\P^2)^* \setminus (O \cup A)$. In other words, these points lie on the same side of the line $A$ in the affine plane $(\P^2)^* \setminus O$, and since this holds for every side $A$ of $P^*$, this precisely means that $P^*$ is convex and contains the point $L$ in its interior.  
\end{proof}

Now, consider a convex polygon $P$ in the affine plane, and let $L$ be the line at infinity. Fix a point~$O$ in the interior of $P$. Then, by Lemma \ref{lemma:convduality}, the dual polygon $P^*$ is convex in the affine chart  $(\P^2)^* \setminus O$ and contains the point $L$ in its interior. 
\begin{proposition}\label{prop:convlim}
For any given $k>0$, the polygon $S^k(P)$ is convex if and only its dual $D^k(P^*)$ contains the point $L \in (\P^2)^*$ in its interior.
\end{proposition}
\begin{proof}
Assume that the polygon $S^k(P)$ is convex. Observe that $D^k(S^k(P)) = P$, so $S^k(P)$ contains the polygon $P$ and hence the point $O$ in its interior. Therefore, since $S^k(P)$ is convex and contains the point $O$, by Lemma \ref{lemma:convduality} the polygon $D^k(P^*) = S^k(P)^*$ contains the point $L$. Conversely, assume that  $D^k(P^*)$ contains the point $L$. Note that since the pentagram map $D$ preserves convexity, and $P^*$ is convex, we also have that $D^k(P^*)$ is convex. Therefore, since $D^k(P^*)$ is convex and contains the point $L$, by Lemma~\ref{lemma:convduality} we have that $S^k(P) = D^k(P^*)^*$ is convex too, as desired.
\end{proof}
\begin{corollary}\label{cor:duallimit}
Each of the polygons $S^k(P)$, where $k > 0$, is convex if and only if the limit point of the sequence  $D^k(P^*)$ is the line at infinity $L$.
\end{corollary}
\begin{proof}
The limit point is the unique point in the intersection $\bigcap_{k > 0} \mathrm{conv}(D^k(P^*))$, so $L \in \mathrm{conv}(D^k(P^*))$ for every $k > 0$ if and only if the limit point is $L$.
\end{proof}
Now, to complete the proof of Theorem \ref{thm1}, it suffices to show that the limit point of the pentagram orbit $D^k(P^*)$ is $L$ if and only if the projective mapping $G_P$ is affine. First, assume that $L$ is the limit point for $D^k(P^*)$. Then, $L$, viewed as a point in the dual plane, is fixed by the mapping $G_{P^*}$ (Proposition~\ref{prop:glick}, Item 3). But in view of the equality $G_{P^*} = G_P^*$  (Proposition \ref{prop:glick}, Item 1), this is equivalent to saying that $L$, viewed as the line in the initial plane, is invariant under $G_P$. So, $G_P$ preserves the line at infinity and hence is affine. Conversely, assume that the mapping $G_P$ is affine. Then the line at infinity $L$, viewed as a point in the dual plane, is a fixed point of $G_{P^*}$. And since $L$ is inside $P^*$, it follows from Proposition \ref{prop:glick}, Item 3 that $L$ must be the limit point of $D^k(P^*)$, as desired. Thus, Theorem \ref{thm1} is proved. 

\begin{remark}\label{rem:ref}
As pointed out by the referee, one can use Corollary \ref{cor:duallimit} to directly prove that the set of polygons which remain convex under iterations of $S$  is of codimension two. The idea is as follows. Let $\mathcal P_n$ be the space of convex $n$-gons. Consider the map $\zeta \colon \mathcal P_n \to (\P^2)^*$ which takes a polygon $P$ to the limit point of the sequence $D^k(P^*)$. Then, by Corollary \ref{cor:duallimit}, the set of polygons which remain convex under iterations of $S$ is the preimage under $\zeta$ of the line at infinity. Assuming the map $\zeta$ is smooth (which follows, for example, from Glick's construction), it must be a submersion because it commutes with the projective group action. So the preimage of the line at infinity (which is a single point in  $(\P^2)^*$) is a codimension two submanifold, as desired.
\end{remark}

\medskip

\section{Case study: pentagons}

\begin{figure}[t]
\centering
\includegraphics[width = 6cm]{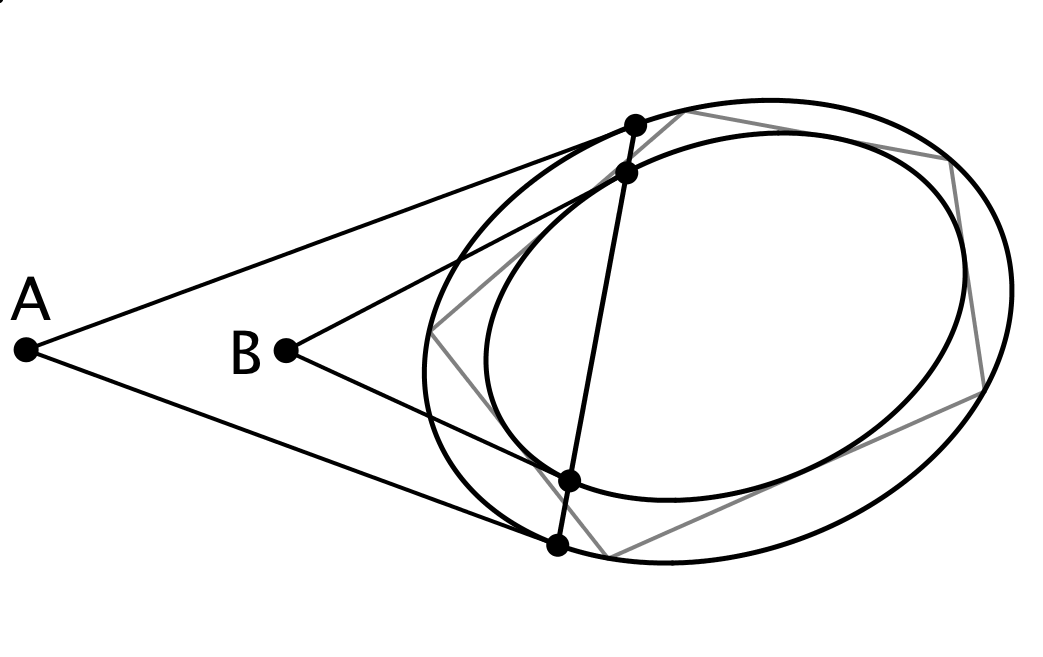}
\caption{To the definition of the map $R_P(A) := B$.}\label{Fig:rmap}
\end{figure}

We now prove the equivalence $1 \Leftrightarrow 4$ (i.e. Theorem \ref{thm2}): for a convex pentagon $P$, all its successive images under the inverse pentagram map $S$ are convex if and only if the conic inscribed in $P$ and the conic circumscribed about $P$ are concentric. The strategy of the proof is as follows. Consider a pentagon $P$, and let its inscribed and circumscribed conics be given by homogeneous quadratic forms $Q_I$ and $Q_C$ respectively. Let $R_P:= Q_I^{-1} Q_C$. Note that since the forms $Q_I$ and $Q_C$ are defined up to a constant factor, $R_P$ is well-defined as a projective map $\P^2 \to \P^2$. The geometric meaning of that map is the following: it takes a point $A \in \P^2$ to point $B \in \P^2$ such that the polar line of $A$ with respect to the circumscribed conic is the same as the polar line of $B$ with respect to the inscribed conic, see Figure~\ref{Fig:rmap}. We will show that the operator $R_P$ has all the same properties as Glick's operator $G_P$. From that it follows that persistence of convexity is equivalent to $R_P$ being affine. But $R_P$ is affine precisely when the inscribed and circumscribed conics are concentric.

To describe the properties of $R_P$, we use the following classical theorem of Kasner. Let $D$, as before, be the pentagram map, and let $I$ be the map which sends a pentagon $P$ to a new pentagon whose vertices are the tangency points of the sides of $P$ and the inscribed conic, see left picture in Figure \ref{Fig:Imap}. 

\begin{figure}[b]
\centering
\includegraphics[width = 6cm]{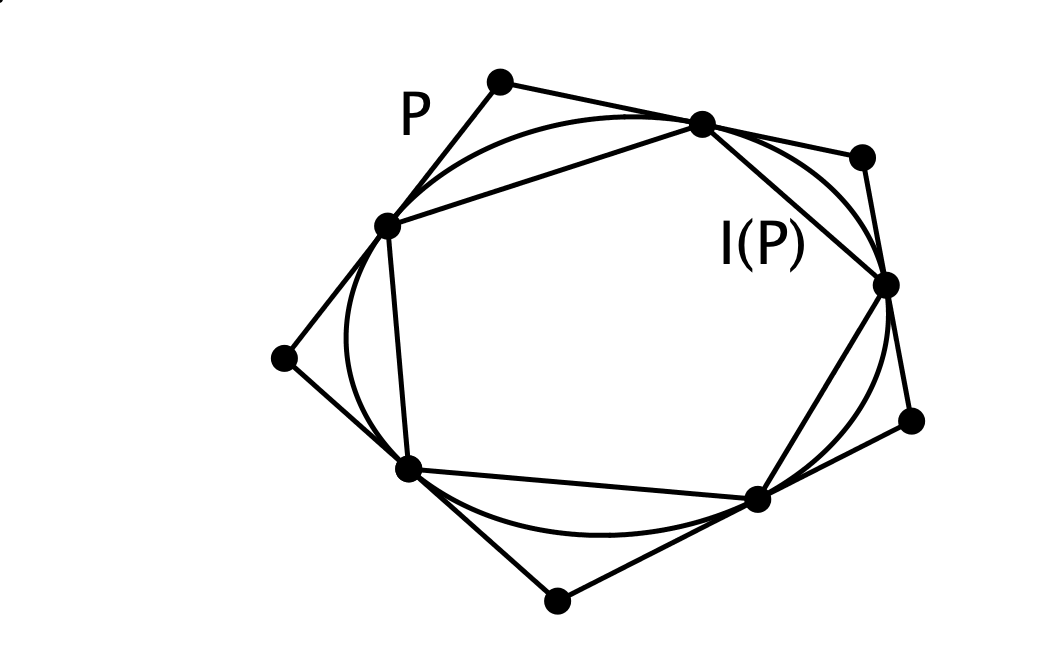} \qquad
\includegraphics[width = 7cm]{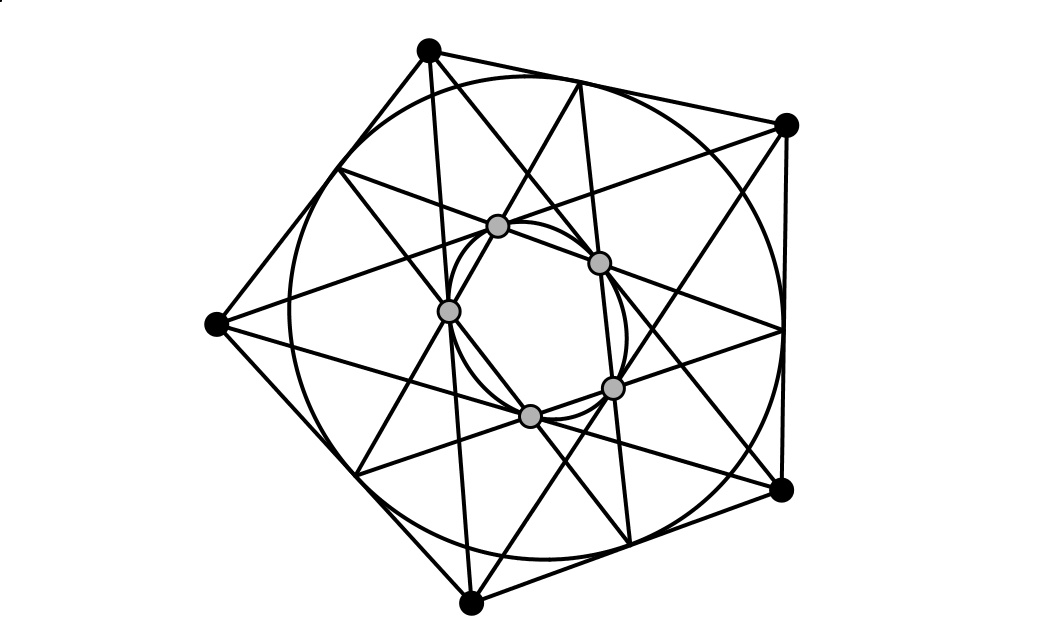}
\caption{Definition of the operation $I$ and Kasner's theorem.}\label{Fig:Imap}
\end{figure}

\begin{theorem}[Kasner \cite{kasner1928projective}] The operations $D$ and $I$ on pentagons commute: $DI = ID$, see right picture in Figure~\ref{Fig:Imap}. \end{theorem}

\begin{figure}[t]
\centering
\includegraphics[width = 7cm]{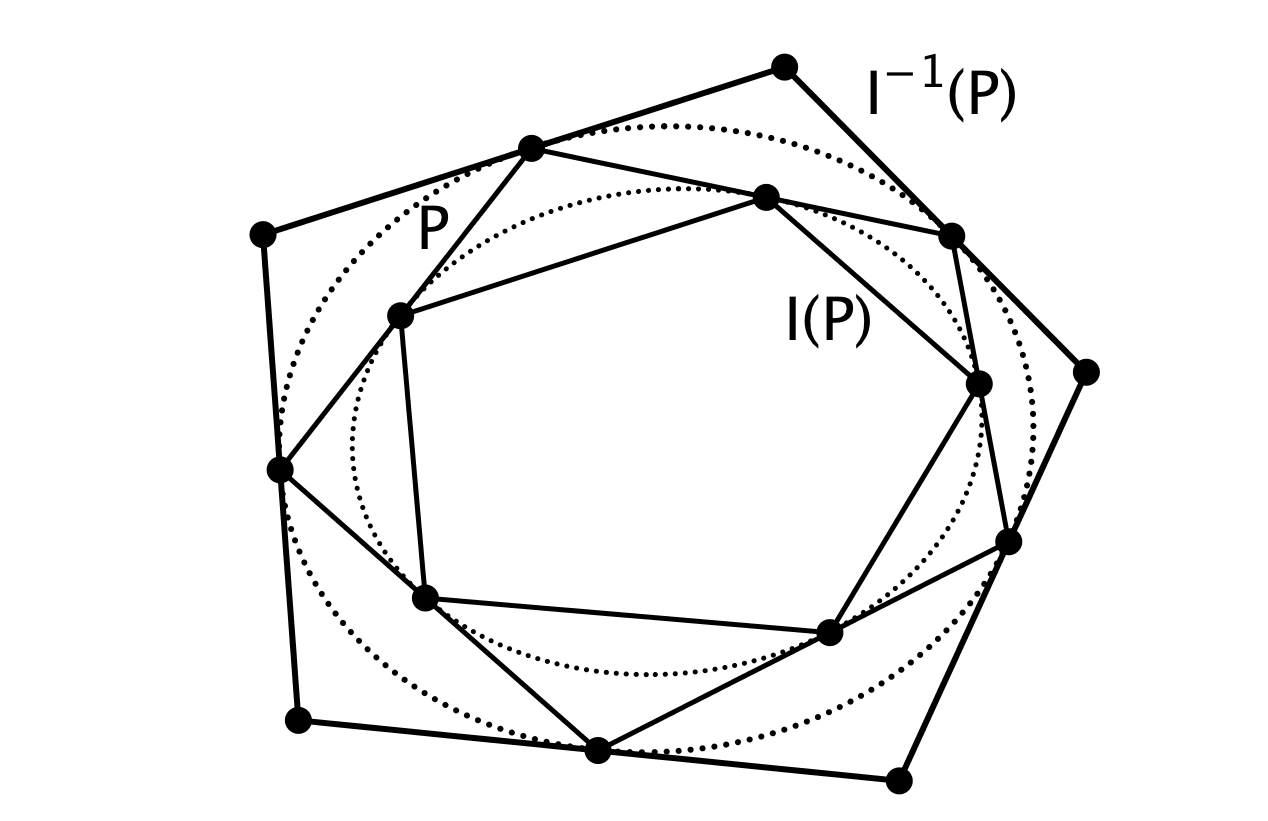}
\caption{To the proof of Corollary \ref{cor:Kasner}.}\label{Fig:Imap2}
\end{figure}
\begin{corollary}\label{cor:Kasner}
For any pentagon $P$, one has $R_{D(P)} = R_P$.
\end{corollary}
\begin{proof}
Consider Figure \ref{Fig:Imap2}. 
Observe that the sides of $P$ are polar to the vertices of $I(P)$ with respect to the inscribed conic and to the vertices of $I^{-1}(P)$ with respect to the circumscribed conic. Therefore, $R_P(I^{-1}(P)) = I(P)$. But since the map $I$ commutes with projective transformations, this is the same as to say that
$
R_P(P) = I^2(P).
$
At the same time, we have
\begin{align*}
R_{D(P)}(P) = D^{-1}(R_{D(P)}(D(P)))  =D^{-1} (I^2(D(P))) = I^2(P),
\end{align*}
where the last equality follows from Kasner's theorem. 
%
Thus, the mappings $R_{P}$ and $R_{D(P)}$ both map $P$ to $I^2(P)$, while $R_{D(P)}^{-1} \circ R_{P}$ maps $P$ to itself. But the only projective automorphism of a generic pentagon is the identity, so we must have $R_{D(P)} = R_P$ for generic, and hence, by continuity, for all polygons.
\end{proof}
\begin{remark}
One can also use a more precise version of Kasner's theorem which takes into account labeling of vertices to show that $R_P$ and $R_{D(P)}$ map every vertex of $P$ to the same vertex of $I^2(P)$ and hence coincide.
\end{remark}
\begin{proposition}\label{prop:glick2}
Let $P$ be a convex pentagon in the affine plane. Then the associated operator $R_P$ has all the properties listed in Proposition \ref{prop:glick}. Namely, the following is true:
\begin{enumerate}
\item For the dual pentagon $P^*$, one has $R_{P^*} = R_P^*$. 
\item One has $R_P(\mathrm{conv}(P)) \subset \mathrm{int}(\mathrm{conv}(P))$. 
\item The limit point 
$
\bigcap_{k \geq 0} \mathrm{conv}(D^k(P))
$
of the pentagram orbit $D^k(P)$ is a fixed point of $R_P$. Moreover, it is the only fixed point of $R_P$ in $\mathrm{conv}(P)$.
\end{enumerate}
\end{proposition}
\begin{proof}
\begin{enumerate}
\item The conic \textit{inscribed} in $P^*$ is the dual of the conic \textit{circumscribed} about $P$, i.e. it is given by the quadratic form $Q_C^{-1}$. Likewise, the conic circumscribed about $P^*$ is given by $Q_I^{-1}$. So,
$$
R_{P^*} = (Q_C^{-1})^{-1}Q_I^{-1} = Q_C Q_I^{-1} =(Q_I^{-1}Q_C)^* = R_P^*.
$$
\item Any point in the convex hull of $P$ is inside the circumscribed conic (i.e. in the contractible component of the complement to that conic in $\P^2$), or belongs to that conic. Therefore, the polar line of that point with respect to the circumscribed conic is either outside or tangent to that conic, and the point polar to that line with respect to the inscribed conic is inside the inscribed conic and hence inside $P$. So, $R_P(\mathrm{conv}(P)) \subset \mathrm{int}(\mathrm{conv}(P))$,
as claimed.
\item
By Corollary \ref{cor:Kasner}, we have $R_P = R_{D^k(P)}$, so
$
R_P( \mathrm{conv}(D^k(P))) = R_{D^k(P)}(\mathrm{conv}(D^k(P)),
$
which, by above, is a subset of $\mathrm{conv}(D^k(P))$.
Therefore, $$R_P\left(\bigcap\nolimits_{k \geq 0} \mathrm{conv}(D^k(P))\right) \subset \bigcap\nolimits_{k \geq 0} R_p(\mathrm{conv}(D^k(P))) \subset  \bigcap\nolimits_{k \geq 0} \mathrm{conv}(D^k(P)),
$$ meaning that the limit point of the pentagram orbit $D^k(P)$ is indeed fixed by $R_P$. As for the uniqueness of the fixed point, it follows from other properties in the same way as for $G_P$. 
 \qedhere
\end{enumerate}
\end{proof}
From this proposition it follows that for a convex pentagon $P$, all its successive images under the map $S$ are convex if and only if the associated projective mapping $R_P$ is affine. Indeed, the proof of the corresponding statement for  $G_P$ is based on Proposition \ref{prop:glick}, and since that proposition also holds for $R_P$, the result follows. Now, it remains to show that $R_P$ is affine  if and only if the conic inscribed in $P$ and the conic circumscribed about $P$ are concentric. To that end, recall that the center of the conic is, by definition, the point polar to the line at infinity with respect to that conic. Therefore, the image of the line at infinity under the map $R_P$ is the line polar with respect to the inscribed conic to the center of the circumscribed conic. Thus, the centers coincide precisely when $R_P$ preserves the line at infinity, i.e. is affine, as desired.

\begin{remark}
Implication $4 \Rightarrow 1$ ($P$ is a convex pentagon with concentric inscribed and circumscribed conics $\Rightarrow$ all pentagons $S^k(P)$ are convex) can also be established as follows. First assume that the inscribed and circumscribed conics of $P$ are \textit{confocal}. Then, as shown in \cite{levi2007poncelet}, there exists an affine transformation $D_P$ taking $P$ to its pentagram image $D(P)$. Therefore, we have $S^k(P) = D_P^{-k}(P)$, and $S^k(P)$ is convex for every $k$. Now assume that the inscribed and circumscribed conics of $P$ are concentric, but not necessarily confocal. Then there exists a (generally speaking, defined over $\C$) affine transformation $Q$ which makes those concentric conics confocal. Then, since $Q(P)$ has confocal  inscribed and circumscribed conics, all the pentagons $S^k(Q(P))$ are convex, and the same holds for $S^k(P) = Q^{-1}(S^k(Q(P)))$, as claimed.  

\end{remark}

\begin{remark}
Instead of proving the equivalence $1 \Leftrightarrow 4$, we could have proved $3 \Leftrightarrow 4$, i.e. $G_P$ is affine $\Leftrightarrow$ $R_P$ is affine, as follows. First, recall that by Clebsch's theorem for any pentagon $P$ there exists a projective transformation $D_P$ such that $D_P(P) = D(P)$, see e.g. \cite[Theorem 2.1]{schwartz1992pentagram}. Furthermore, any pentagon is projectively equivalent to its dual, see e.g. \cite[Proposition 5]{fuchs2009self}. So, since $I(P)$ is polar to $P$ and hence projective to $P^*$, there exists a projective transformation such that $I(P) = I_P(P)$.  It is then a direct corollary of Kasner's theorem that  transformations $D_P$ and $I_P$ commute, cf. \cite[Corollary 7]{pamfilos2016diagonal}. Therefore, $D_P$ commutes with $R_P = I_P^2$. Further, as observed in \cite{glick2020limit}, $D_P$ coincides, as a linear operator in $3$-space, with $G_P - 3\cdot\Id$, so $R_P$ commutes with $G_P$. Now, assume that $R_P$ is affine. Then the line at infinity $L$ is a fixed point of the dual operator $R_{P}^*$. Moreover, by Proposition~\ref{prop:glick2}, Item 3, it is the only fixed point of $R_{P}^*$ in the interior of $P^*$. Then, since $G_P^*$ commutes with $R_P^*$ and preserves the interior of $P^*$, it must preserve $L$, which means that $G_P$ is affine. Likewise, if $G_P$ is affine, then $R_P$ is affine too, as desired.
\end{remark}

%

%

\bibliographystyle{plain}
\bibliography{pent.bib}

\end{document}